\documentclass[12pt]{article}
\usepackage{amsmath}
\usepackage{amsthm}
\usepackage{amssymb}
\usepackage[T1]{fontenc}

\newtheorem{tw}{Theorem}
\newtheorem{cor}[tw]{Corollary}
\newtheorem{lem}[tw]{Lemma}
\newtheorem{df}[tw]{Definition}
\newtheorem{prop}[tw]{Proposition}
\newtheorem{claim}[tw]{Claim}

\begin{document}
\title{There is no equivariant coarse embedding of $L_{p}$ into $\ell_p$.}
\author{Krzysztof {\'S}wi\k{e}cicki }
\maketitle
\begin{abstract}
In this paper we prove that $L_{p}$ does not admit an equivariant coarse embedding into $\ell_p$ i.e there is no proper, affine, isometric action of $L_{p}$, viewed as a group under addition with the standard metric $|| . ||_p$, on $\ell_p$. This is done by showing that representations of $L_{p}$  into $ Isom(\ell_p)$ has to be trivial, which allows us to reduce the question to bi-Lipschitz setting.
\end{abstract}

\section{Introduction}
Both $\ell_p$ and $L_{p}$ spaces are some of the oldest examples of Banach spaces so the question of their embeddability in different categories was raised as soon as people started to study Banach spaces in those settings. Let us look at a brief history of embeddability of $L_{p}$ into $\ell_p$.

Let's first look at the linear theory - the fact that there is no bounded linear embedding from the classical fact that $\ell_2$ is linearly isometric to a subspace of $L_p$ for any $1 \leq p \neq 2 < \infty$ but does not bi-Lipschitzly embed into $\ell_p$.     
More can be said about the linear theory from a more general theorem of Kadets and Pełczyński \cite{KP} which establishes a dichotomy for a closed subspace $X$ of $L_p$ when $2 < p < \infty$. Namely $X$ is either isomorphic to the Hilbert $\ell_2$ space or $X$ contains a subspace that is isomorphic to $\ell_p$ and complemented in $L_{p}$.

A big step in the study of Banach spaces was when people started to look at them from a purely metric point of view and forgot the underlying linear structure.
A fundamental result in this setting is a theorem by Mazur and Ulam (see \cite{MU}) that studies isometries $f \colon X \to Y$ between real normed spaces $X, Y$, that maps $0$ to $0$. By looking at midpoints of line segments they prove that $f$ is in fact a bounded linear operator. Thus the problem of metric embedding of $L_{p}$ into $\ell_p$ can be reduced to the linear setting, which gives us a negative answer. 

The standard proof of nonexistence of bi-Lipschitz embedding from $L_p$ into $\ell_p$ follows Heinrich and Mankiewicz who in \cite{HM} proved more generally that if $X \colon \to Y$ is a bi-Lipschitz embedding of separable Banach space $X$ into a space $Y$ with the Radon–Nikodym property, then there exists a point of Gâteaux differentiability of $f$ and the derivative at that point is a linear embedding with distortion bounded by distortion of $f$.
Their proof is relying on the Rademacher’s differentiability theorem, which says that every Lipschitz map $f \colon X \to Y$ from a separable Banach space $X$ into a Banach space $Y$ with the Radon–Nikodym property is differentiable Haar-almost-everywhere. 

More general type of maps that one may consider is quasi-isometric (or coarse Lipschitz) embeddings where linear bounds of the distance between any two points in replaced by an affine bound. In this setting it was Kalton and Randrianarivony who proved in \cite{KR} that $L_{p}$ does not admit a quasi-isometric embedding into $\ell_p$ by
studying graphs $G_k(\mathbb{M}) = \{ \overline{n}=(n_1, \dots ,n_k) \colon n_i \in \mathbb{M}, n_1 < n_2 < \dots < n_k\}$ with a metric $d(\overline{n}, \overline{m}) = | \{i: n_i \neq m_i\} |$ for any subset $\mathbb{M}$ of the natural numbers. More precisely they estimate the minimal distortion of any Lipschitz embedding of $G_k(\mathbb{M})$ into  $\ell_p$-like Banach spaces using Ramsey’s Theorem.

The most general metric embedding we're going to consider is a coarse category, where distance bounds we consider are arbitrary nondecreasing functions. In this setting question about embeddability of $L_{p}$ into $\ell_p$ has only partial answers. Namely for 
$1 \leq p < 2$ it follows by work of by Mendel and Naor in \cite{MN} that $L_{p}$ does embed into $\ell_p$.
This is done by factoring the embedding through Hilbert space, based on the argument by Nowak who proved in \cite{N} that $\ell_2$ coarsely embeds into $\ell_p$ for $1 \leq p < \infty$.

Changing the setting from metric to more general it is worth noting that Johnson, Lindenstrauss and Schechtman proved in \cite{JLS} that if Banach space $X$ is uniformly homeomorphic to $\ell_p$ for $1 < p < \infty$ then $X$ is linearly isomorphic to $\ell_p$.
More recently Mendel and Naor in \cite{MN2}[Theorem 1.10] showed that uniform homeomorphism from $L_p$ into $\ell_q$ exists if and only if $p \leq q$ or $q \leq p \leq 2$. This was done by developing a theory of metric cotype, which is an extension of the Rademacher’s cotype to a purely metric setting. 

In this paper, we focus on the category of equivariant coarse embeddings - which are coarse embeddings that satisfy a certain compatibility condition with a predetermine representation into isometry group of the target space. We find that the existence of such an embedding of a normed vector space into $\ell_p$ forces it to be linearly isomorphic to $\ell_p$. It is important to note result by Cornulier, Tessera, Valette \cite{CTV} that if locally compact, compactly generated, amenable group $G$ coarsely embeds into the Hilbert space, then it also embeds in the coarsely equivariant way. Since we view $L_p$ as a topological group under addition, one can hope it might be possible to generalize this fact in order to use our result to attack the general question of coarse embeddability of $L_{p}$ into $\ell_p$.

\section{Preliminaries}
In this section we recall some basic definitions starting with some group theoretic vocabulary: 

\begin{df}
    For a set $X$ we denote by $Sym(X)$ group of all bijections of $X$.
\end{df}

\begin{df}
    We say that that a group element $g$ is a n-th root of a group element $h$ if $g^n=h$
\end{df}

Note that if a group $G$ acts on $X$ and $g \in G$ fixes $x \in X$ then $g^n$ also fixes $x$ for any natural number $n$. Hence we established the following claim.

\begin{claim}
If $h$ is a n-th root of a group element $g$ acting on $X$  $supp(g) \subseteq supp(h)$.
\end{claim}

We now recall definitions of Banach spaces that we will study:

\begin{df}
For $1<p<\infty$ we define $\ell_p$ as space of all bi-infinite sequences $\{x_i\}_1^{\infty}$ such that $\sum _{i \in \mathbb{Z}}|x_{i}|^{p}<\infty$ equipped with the norm $\|\{x_i\}_1^{\infty}\|_p = \sum _{n}|x_{n}|^{p}$.
\end{df}

\begin{df}
For $1<p<\infty$ we define $\mathcal{L}_{p}$ to be a space of all measurable functions $f: \mathbb{R} \to \mathbb{R} $ such that  ${\displaystyle \|f\|_{p}\equiv \left(\int _{S}|f|^{p}\;\mathrm {d} \mu \right)^{1/p}<\infty }$. Further more we denote by $L_{p}$ the quotient space with respect to the kernel of $|| . ||_p$, which defines a complete note on the said quotient.
\end{df}
Now we introduce different types of embeddings that we'll consider:

\begin{df}
A map $f\colon (X,d_X) \to (Y,d_Y) $ is an isometry if it preserves distance i.e. for any $x_1, x_2 \in X$ the following equality holds
$ d d_{Y}(f(x_{1}),f(x_{2}))= d_{X}(x_{1},x_{2}).$ 
\end{df}
A set of all bijective isometries of a fixed metric space $(X,d_X)$ forms a group under composition that we are denoting by $Isom(X)$. A classical result by Banach (see \cite{Ban}) states that the only isometries of the space $\ell_p$ are permutations of the support of the sequence (which is corresponding to $Sym(\mathbb{Z} \subset Isom(\ell_p)$) and changes of signs of its elements (every change of sign is corresponding to a different copy of $\mathbb{Z_2}$ and since there are $\mathbb{Z}$ many of them we conclude that $\mathbb{Z}_2^{\mathbb{Z}} \subset Isom(\ell_p) $). Since a permutation of the support does not commute with a change of sign we conclude that the group structure of $Isom(\ell_p)$ is a semi-direct product of $Sym(\mathbb{Z})$ with $\mathbb{Z}_2^{\mathbb{Z}}$ where $Sym(\mathbb{Z})$ is acting by change of support of $\mathbb{Z}_2^{\mathbb{Z}}$.
By a standard fact in group theory this is equivalent to the existence of the following short exact sequence:
		\begin{equation*}
		1 \to \mathbb{Z}_2^{\mathbb{Z}} \xrightarrow{i} Isom(\ell_p) \xrightarrow{p} Sym(\mathbb{Z}) \to 1.
		\end{equation*}
		
Now we introduce different generalizations of isometric maps:
		
\begin{df}
A map $f\colon (X,d_X) \to (Y,d_Y) $ is a a bi-Lipschitz embedding if there exist constants d,D such that for any $x_1, x_2 \in X$ the following inequality holds
$ d d_{Y}(f(x_{1}),f(x_{2}))\leq d_{X}(x_{1},x_{2})\leq D d_{X}(x_{1},x_{2}).$ Constant D is called distortion of $f$ and is denoted $\mathrm{Dist}(f)$.
\end{df}

\begin{df}
A map $f\colon (X,d_X) \to (Y,d_Y) $ is a a coarse embedding if there exist nondecreasing functions $\rho_{-}, \rho_{+} \colon [0, \infty) \to [0, \infty)$ such that ${\displaystyle \lim _{t\to \infty}\rho_{-}(t) = \infty}$ and for any $x_1, x_2 \in X$ the following inequality holds $\rho_{-} (d_{Y}(f(x_{1}),f(x_{2})) ) \leq d_{X}(x_{1},x_{2}) \leq \rho_{+} ({d_{Y}(f(x_{1}),f(x_{2}))})$.
\end{df}

\begin{df}
Let $V$ be a normed, vector space.  A linear bijection $U\colon V \to V$ is called a linear isometry if it preserves the norm i.e. for every $v \in V$ we have $U(\lVert v \rVert_{V}) = \lVert v \rVert_{V}$.
\end{df}

Note that if $G$ is a group then having it act in an isometric way on $V$ is the same as defining an isometric representation i.e. a homomorphism $\pi \colon G \to Isom(V)$ into the group of all linear isometries of $V$

\begin{df}
An affine isometry of a normed vector space $V$ is a map $A\colon V \to V$ such that for every $v \in V$  $A(v) = U(v) +b$, where $U$ is a linear isometry and $b$ is a fixed vector in $V$.
\end{df}

We say that group $G$ acts on $V$ by an affine isometries if for every $g \in G$ there exists an affine isometry $A_g \colon V \to V$ such that $A_{gh}=A_g A_h$. By our previous remark $A_g(v)= \pi_g v + b_g$ where $\pi \colon G \to Isom(V)$ and $b \colon G \to V$.If we rewrite $A_{gh}=A_g A_h$ in this form we reach identity known as the cocycle condition:
\begin{equation}\label{coc}
    b_{gh}=\pi_g b_h +b_g
\end{equation}

We say that affine action of $G$ on $V$ is proper if $lim _{|g | \to \infty} \|b_g\|_V = \infty$. It is important to remark here that in the proof of our result we will be suing the standard $|| . ||_p$ norm as a length function on $L_{p}$. Elementary calculation (see \cite{NY} for a proof) yields:

\begin{prop}
    Let $G$ be a finitely generated group, which admits a proper, affine, isometric action on a normed vector space $V$, with a cocycle b. Then b is a coarse embedding.
\end{prop}

Because of the above admitting a proper, affine, isometric action on $V$ is viewed as a stronger version of a coarse embedding and is also called an equivariant coarse embedding.

\section{Proof of the main result}

	\begin{lem}\label{l1}
		Let $Sym(\mathbb{Z})$ denote the group of bijections of integers. If $\sigma \in Sym(\mathbb{Z})$ has a n-th root $\delta$ (i.e. $\delta^n=\sigma$) then there exists a unique n-th root $\sqrt[n]{\sigma}$ s.t. $supp(\sqrt[n]{\sigma}) = supp(\sigma)$.
	\end{lem}
	
	\begin{proof}
		 Let $k,l \in supp(\sigma)$ s.t $\sigma(k)=l$ and $m \in supp(\delta) - supp(\sigma)$. Assume that $\delta(k)=m$. Notice that $\sigma$ and $\delta$ commute, since $\sigma$ belongs to a cyclic subgroup generated by $\delta$. Thus $\sigma(\delta (k)) = \sigma(m)=m$ and $\delta (\sigma (k)) = \delta(l)$ should be equal. But $\delta$ is an isomorphism sending $k$ to $m$, so it can not send $l$ to $m$ as well. Contradiction.

		We just showed that for any $k \in supp(\sigma)$, $\delta(k)$ also belongs to $supp(\sigma)$. It mean that all cycles in a cycle decomposition of $\delta$ contain either only elements of $supp(\sigma)$ or only those from $supp(\delta) - supp(\sigma)$. After removing later cycles from $\delta$ we obtain $supp(\sqrt[n]{\sigma})$.
	\end{proof}
		
	\begin{lem}\label{l2}
		Let $(V,+)$ be a vector space viewed as an abelian group under addition. Then every homomorphism $\sigma:(V,+) \to Sym(\mathbb{Z})$ is trivial.
	\end{lem}

	\begin{proof}
		For any $v \in V$ and natural number $n$ we have $\sigma(\frac{v}{n})^n = \sigma(v)$, hence by Lemma \ref{l1} $\sqrt[n]{\sigma(v)}$ always exist. We will show that $\sigma(v)=e$. Consider two cases.
		
		First assume that all elements of $supp(\sigma(v))$ have a finite orbit. Let $k \in supp(\sigma(v))$ and its orbit consist of n integers. Then there exists $\sqrt[n!]{\sigma(v)}$,
		which sends this n-tuple of integers to itself. Order of all elements of $S_n$ is divides $n!$, so $\sqrt[n!]{\sigma(v)}^{n!} = \sigma(v)$ acts on that n-tuple trivially. Thus $k \notin supp(\sigma(v))$, contradiction.
		
		Now let $k \in supp(\sigma(v))$ have an unbounded orbit. By Lemma \ref{l1} $\sqrt{\sigma(v)}$ sends $k$ to some $l \neq k$, which belongs to the orbit of $k$ under $\sigma(v)$. It follows that $\sqrt{\sigma(v)}(l)=\sigma(v)(k)$. We claim that $\sqrt{\sigma(v)}(\sigma(v)^i(l)) = \sigma(v)^{i+1}(k)$ for any integer i. Case $i=0$ is our basis of induction.
		Assume $i>0$ and $\sqrt{\sigma(v)}(\sigma(v)^i(l)) = \sigma(v)^{i+1}(k)$. Then $\sigma(v)(\sqrt{\sigma(v)}(\sigma(v)^i(l)))=\sigma(v)(\sigma(v)^{i+1}(k))=\sigma(v)^{i+2}(k)$ is equal to $\sqrt{\sigma(v)}(\sigma(v)(\sigma(v)^i(l)))=\sqrt{\sigma(v)}(\sigma(v)^{i+1}(l)))$, proving inductive step.
		Similarly if $i<0$ and $\sqrt{\sigma(v)}(\sigma(v)^i(l)) = \sigma(v)^{i+1}(k)$. Then $\sigma(v)(\sqrt{\sigma(v)}(\sigma(v)^{i-1}(l)))$ is equal to $\sqrt{\sigma(v)}(\sigma(v)(\sigma(v)^{i-1}(l)))=\sqrt{\sigma(v)}(\sigma(v)^{i}(l)))$, finishing the proof of the claim.
		%new part
		Since $k=\sigma(v)^i(l)$ for some $i \neq 0$ we have $\sqrt{\sigma(v)}(k) = \sqrt{\sigma(v)} \sigma(v)^{i}(l) =  \sigma(v)^{i}(k)$, which gives us that $i$ is an integer such that $2i=1$.
		Contradiction with the assumption that $\sqrt{\sigma(v)}(k)=l$. Thus $\sigma(v)$ can not be of this form.		
	\end{proof}	
	
	\begin{tw}\label{rep}
		All representations $\pi:(V,+) \to Isom(\ell_p)$ are trivial for $p \neq 2$.
	\end{tw}
	
	\begin{proof}
		Recall that in section 2 we have established the existence of the following short exact sequence:
		\begin{equation*}
		1 \to \mathbb{Z}_2^{\mathbb{Z}} \xrightarrow{i} Isom(\ell_p) \xrightarrow{p} Sym(\mathbb{Z}) \to 1.
		\end{equation*}
		By Lemma \ref{l2} homomorphisms $p \circ \pi :(V,+) \to Sym(\mathbb{Z})$ is trivial, so $\pi(V) \leq ker(p) \cong im(i) \cong \mathbb{Z}_2^{\mathbb{Z}}$. But every homomorphism $\rho:(V,+) \to \mathbb{Z}_2^{\mathbb{Z}}$ must be trivial, since $\mathbb{Z}_2^{\mathbb{Z}}$ is a torsion group. Thus $\pi$ also needs to be trivial.
	\end{proof}	
	
	\begin{tw}
		Every normed, vector space $(V,+)$ admitting a proper, affine, isometric action on $\ell_p$ is isomorphically embeddable into $\ell_p$.
	\end{tw}
	
	\begin{proof}
		By Theorem \ref{rep} linear representation of $V$ is trivial. The cocycle condition \ref{coc} gives us then the existence of an additive, coarse embedding $A: V \to l_p$ i.e. $\rho_{-} (\lVert v \rVert_{V} ) \leq \lVert A(v) \rVert_{l_p} \leq \rho_{+} (\lVert v \rVert_{V})$.
		Now let $\alpha \in \mathbb{R_{+}}$ be such that $\rho_{-}(\alpha)>0$.
		
		There exists $n \in \mathbb{N}$ such that
	    $2^n\alpha \leq \lVert v \rVert \leq 2^{n+1}\alpha$. Then 
	    $\lVert A(v) \rVert =2^{n} \lVert \frac{A(v)}{2^n} \rVert \geq 2^n \rho_{-}(\alpha) \geq \frac{\rho_{-}(\alpha)}{2\alpha} \lVert v \rVert$.
		Similarly
		$\lVert A(v) \rVert =2^n \lVert \frac{A(v)}{2^n} \rVert \leq 2^n \rho_{+}(2\alpha) \leq \frac{\rho_{+}(2\alpha)}{\alpha} \lVert v \rVert$.
	\end{proof}	
	
	\begin{cor}
		There is no equivariant coarse embedding of $L_p$ into $\ell_p$.
	\end{cor}

\end{document}